\documentclass[11pt, a4paper]{article}

\usepackage[utf8]{inputenc}
\usepackage[french, english]{babel}
\usepackage{indentfirst}
\usepackage[T1]{fontenc}
\usepackage{amsmath, amssymb, amsthm}
\usepackage{url}
\usepackage{hyperref}
\usepackage{geometry}
\usepackage{color}
\usepackage{dsfont}
\usepackage{todonotes}
\usepackage{enumitem}
\usepackage[mathscr]{euscript}
\usepackage{mathtools}
\usepackage[ddmmyyyy,hhmmss]{datetime}
\usepackage{xcolor}

\newcounter{moncompteur}

\newtheorem{theorem}[moncompteur]{Theorem}

\renewcommand\leq{\leqslant}
\renewcommand\geq{\geqslant}
\renewcommand\theta{\vartheta}

\renewcommand\epsilon{\varepsilon}

\begin{document}
%\timestamp
\selectlanguage{english}
\begin{center}\Large Highest perfect power of a\\ product of integers less than $x$\\
\selectlanguage{french}
\small\bsc{Élie Goudout}\end{center}
\selectlanguage{english}
\begin{abstract}
For $x\geq 3$, we define $w(x)$ as the highest integer $w$ for which there exist integers $m, y\geq 1$ and $1\leq n_1<\dots<n_m\leq x$ such that $n_1\cdots n_m=y^w$. We show that
\[w(x)=x\exp\big(-(\sqrt{2}+o(1))\sqrt{\log x\log\log x}\big).\]
\end{abstract}

\indent\indent\footnotesize \emph{Keywords}: Perfect powers; friable numbers.\\
\indent\indent\emph{MSC}: 11A05, 11N56\\

\normalsize
In~\cite{podibhp}, Ska\l ba defines, for $x\geq 3$, $w(x)$ as the highest integer $w$ for which there exist integers $m, y\geq 1$ and $1\leq n_1<\dots<n_m\leq x$ such that $n_1\cdots n_m=y^w$, and shows
\[x\exp\big(-(\sqrt{2}+o(1))L(x)\big)\leq w(x)\leq x\exp\big(-(\textstyle\frac{1}{\sqrt{2}}\displaystyle+o(1))L(x)\big),\]
where $L(x):=\sqrt{\log x\log\log x}$, as $x\rightarrow+\infty$. In this paper, we improve the upper bound, making the result optimal up to the~$\exp\left(o(L(x)\right))$.

\begin{theorem}
When $x\rightarrow +\infty$,
\[w(x)=x\exp\big(-(\sqrt{2}+o(1))L(x)\big).\]
\end{theorem}

\begin{proof}
For $3\leq z\leq x$, let $\Psi(x,z):=\#\left\{n\leq x\,:\quad P^+(n)\leq z\right\}$ be the number of $z$-friable integers less than $x$. From~\cite{Tenenbaum}, we have the classical estimate
\begin{equation}\label{fri}
\Psi(x,z)\ll xu^{-u}+\sqrt{x}.\hspace{2cm}(u=\log x/\log z)
\end{equation}

From now on, let $x\geq 3$, $m$ and $n_1,\dots,n_m\leq x$ denote integers such that $n_1\cdots n_m$ is of the form $y^{w(x)}$. Let $q:=P^+(y^{w(x)})$ denote the highest prime factor of $y^{w(x)}$ and $v_q(y^{w(x)})$ denote its $q$-valuation. Since the number of $q$-friable integers $n\leq x$ with $q\vert n$ is less than $\Psi(x/q,q)$, we have
\[w(x)\leq v_q(y^{w(w)})\leq\Psi(x/q,q)\frac{\log x}{\log q}.\]
The estimate~$(\ref{fri})$ yields that the maximum of the right member is attained for 
\[q=\exp\big((\textstyle\frac{1}{\sqrt{2}}\displaystyle+o(1))L(x)\big).\]
In this case, we get the desired upper bound, which, together with Ska\l ba's lower bound, gives the result.
\end{proof}
\bibliographystyle{plain-fr}
\bibliography{Bibliographie}

\begin{thebibliography}{1}
\expandafter\ifx\csname fonteauteurs\endcsname\relax
\def\fonteauteurs{\scshape}\fi

\bibitem{podibhp}
M.~\bgroup\fonteauteurs\bgroup Ska{\l}ba\egroup\egroup{} :
\newblock Product of dinstinct integers being high powers.
\newblock {\em International Journal of Number Theory},
  13(8)\string:\penalty500\relax 2093--2096, 2017.

\bibitem{Tenenbaum}
G.~\bgroup\fonteauteurs\bgroup Tenenbaum\egroup\egroup{} :
\newblock {\em Introduction {\`a} la th{\'e}orie analytique et probabiliste des
  nombres}.
\newblock Belin, troisi{\`e}me \'edition, octobre 2015.

\end{thebibliography}

\noindent\bsc{Institut de Math\'ematiques de Jussieu-PRG, Universit\'e Paris Diderot,
Sorbonne Paris Cit\'e, 75013 Paris, France}\\

\noindent\textit{E-mail :} \url{eliegoudout@hotmail.com}

\end{document}